\documentclass[12pt]{article}

\usepackage{amsmath}
\usepackage{amsfonts}
\usepackage{latexsym}
\usepackage{amssymb}
\usepackage{amsthm}
\usepackage[normalem]{ulem}
\usepackage{csquotes}
\usepackage{lineno,xcolor}
\usepackage[normalem]{ulem}
\usepackage{filecontents}
\usepackage{scalefnt}
\usepackage{tikz}
\usetikzlibrary{arrows}
\usetikzlibrary{positioning}

\usepackage{lineno,xcolor}
\definecolor{amaranth}{rgb}{0.9, 0.17, 0.31}
\definecolor{bluegray}{rgb}{0.4, 0.6, 0.8}

\makeatletter

\newtheorem*{maintheorem*}{Main Theorem}
\newtheorem{theorem}{Theorem}[section]
\newtheorem{proposition}[theorem]{Proposition}
\newtheorem{corollary}[theorem]{Corollary}
\newtheorem{lemma}[theorem]{Lemma}

\newtheorem*{theorem*}{Theorem}

\newtheorem{remark}[theorem]{Remark}

\newtheorem*{example*}{Example}
\newtheorem*{conjecture*}{Conjecture}
\def\1{\mathbf 1}

\def\S{\mathbf S}
\def\s{\mathbf s}
\def\T{\mathbf T}
\def\t{\mathbf t}

\def\X{\mathbf X}

\def\Y{\mathbf Y}

\def\0{\mathbf 0}

\def\cC{\mathcal C}

\def\cE{\mathcal E}

\def\cH{\mathcal H}

\def\cO{\mathcal O}
\def\cP{\mathcal P}
\def\cX{\mathcal X}

\def\cS{\mathcal S}
\def\cQ{\mathcal Q}

\def\cW{\mathcal W}

\def\PG{{\rm PG}}
\def\PGL{{\rm PGL}}

\def\SL{{\rm SL}}
\def\PSL{{\rm PSL}}
\def\PSp{{\rm PSp}}
\def\PO{{\rm PO}}
\def\mP{{\rm P}}

\def\PGU{{\rm PGU}}
\def\Fq{\mathbb F_{q}}

\def\Fqr{\mathbb F_{q^r}}

\def\Fqt{\mathbb F_{q^2}}
\def\Fqf{\mathbb F_{q^{4}}}

\def\Im{{\rm Im }}

\def\Rad{{\rm Rad }}

\def\Tr{{\rm Tr}}

\def\<{\langle}
\def\>{\rangle}
\newcommand\comment[1]{}

\newcommand*{\shifttext}[2]{
  \settowidth{\@tempdima}{#2}
  \makebox[\@tempdima]{\hspace*{#1}#2}
}

\newcommand\redsout{\bgroup\markoverwith{\textcolor{amaranth}{\rule[0.5ex]{2pt}{0.4pt}}}\ULon}

\title{On the isomorphism of certain primitive\\ $Q$-polynomial  not $P$-polynomial \\ association schemes}

\author{Giusy Monzillo\footnote{The research was supported by the Italian National Group for Algebraic and Geometric Structures and their Applications (GNSAGA-INdAM). } \\
\small {\tt giusy.monzillo@unibas.it}\\
\and
Alessandro Siciliano \\
\small{\tt alessandro.siciliano@unibas.it}\\[0.8ex]
\small Dipartimento di Matematica, Informatica ed Economia\\[-0.8ex]
\small Universit\`a degli Studi della Basilicata\\[-0.8ex]
\small Viale dell'Ateneo Lucano 10 - 85100 Potenza (Italy)
\small Potenza, Italy\\
}

\begin{document}

\date{}

\maketitle

\begin{abstract}
In 2011,   Penttila and Williford constructed an infinite new  family of primitive $Q$-polynomial 3-class  association schemes, not arising from distance regular graphs, by exploring  the geometry of the lines of the unitary polar space $H(3,q^2)$, $q$ even, with respect to a symplectic polar space $W(3,q)$  embedded in  it. 

In a private communication to Penttila and Williford, H.~Tanaka pointed out that these schemes    have the same parameters as the 3-class  schemes  found by Hollmann and Xiang in 2006 by considering  the action of $\PGL(2,q^2)$, $q$ even, on a  non-degenerate conic of $\PG(2,q^2)$ extended in $\PG(2,q^4)$.  Therefore, the question arises whether the above association schemes are isomorphic.
In this paper we provide the positive answer. As by product, we get an isomorphism of strongly regular graphs.

\end{abstract}

{\it Keywords: Association scheme, Finite geometry, Hemisystem}             

{\it Math. Subj. Class.: 05E30, 51E20} 

\section{Introduction}

Let $\mathfrak X=(X,\{R_i\}_{0\le i\le d})$ be a (symmetric) association scheme with $d$ classes. For $0\le i\le d$, let $A_i$ be the adjacency matrix of the relation $R_i$, and $E_i$ the $i$-th primitive idempotent of the Bose-Mesner algebra of $\mathfrak X$ which projects on the $i$-th maximal common eigenspace of $A_0,\ldots,A_d$. The  matrices $P$ and $Q$ defined by
\[
(A_0\ A_1\ \ldots \ A_d)=(E_0\ E_1\ \ldots \ E_d)P
\]
and 
\[
(E_0\ E_1\ \ldots \ E_d)=|X|^{-1}(A_0\ A_1\ \ldots \ A_d)Q
\]
are the {\em first} and the {\em second eigenmatrix} of $\mathfrak X$, respectively. 

An association scheme is said to be  $P${\em -polynomial}, or {\em metric}, if, after a reordering of the relations, there are polynomials $p_i$ of degree $i$ such that $A_i = p_i(A_1)$; an association scheme is called $Q${\em -polynomial}, or {\em cometric}, if, after a reordering of the eigenspaces, there are polynomials $q_i$ of degree $i$ such that $E_i =q_i(E_1)$, where multiplication is done entrywise. The reader is referred to \cite{bi,bcn} for further information on association schemes.

Two association schemes $\mathfrak X=(X,\{R_i\}_{0\le i\le d})$ and $\mathfrak X'=(X',\{R'_i\}_{0\le i\le d})$ are {\em isomorphic} if there exists a bijection $\varphi$ from $X$ to $X'$ such that for each $i\in\{0,\ldots,d\}$ there exists $j\in\{0,\ldots,d\}$ satisfying $\{(\varphi(x),\varphi(y)): (x,y)\in R_i\}=R'_j$; the mapping $\varphi$ is called an  {\em isomorphism} from $\mathfrak X$ to $\mathfrak X'$.
 
The idea of $P$-polynomial and $Q$-polynomial schemes was introduced by Delsarte in \cite{del}, who observed a formal duality between the two notions. Delsarte also noted that $\mathfrak X$ is  $P$-polynomial if and only if, after a proper re-ordering of the relations,  $(X,R_1)$ is  a distance-regular graph \cite[Theorems 5.6 and 5.16]{del}.   
On the other hand, $Q$-polynomial schemes which are neither $P$-polynomial nor duals of $P$-polynomial schemes seem to be quite rare. In \cite{vdmm}  van Dam, Martin and  Muzychuk  constructed an infinite family of such schemes  from hemisystems of the unitary polar space $H(3,q^2)$  provided in \cite{w}. In 2011, Penttila and Williford \cite{pw} constructed another infinite family of  $Q$-polynomial 3-class  association schemes, not $P$-polynomial nor the dual of a $P$-polynomial, by considering a relative hemisystem of $H(3,q^2)$, $q$ even, with respect  to a symplectic polar space $W(3,q)$ embedded in  it. These schemes differ from all those previously known, they being primitive.  The known examples of $Q$-polynomial schemes which are not $P$-polynomial are listed in   \cite{mmw,w}.

We underline that the  Penttila-Williford $3$-class schemes are obtained by applying  \cite[Theorem 2]{pw}  which provides primitive $Q$-polynomial  subschemes of $Q$-polynomial $Q$-bipartite  schemes defined on certain generalized quadrangles. This result can be viewed as a reversal of the so-called \textquote{extended $Q$-bipartite double} construction given  in  \cite{mmw}. On the other hand, looking at the Krein array of the  generic Penttila-Williford scheme, we may note that it comes from  a strongly regular graph after splitting one of its relations in two.

In a private communication to the authors of \cite{pw}, H.~Tanaka pointed out that their 3-class schemes  have the same parameters as the 3-class schemes provided by Hollmann and Xiang in \cite{hx}. The latter, which were previously not noticed to be  $Q$-polynomial, are obtained as  fusion of association schemes constructed from the action of the projective group $\PGL(2,q^2)$, $q$ even,  on a non-degenerate conic in the Desarguesian projective plane $\PG(2,q^2)$ extended in $\PG(2,q^4)$.

Therefore the question arises whether there exists an isomorphism that takes the Pentilla-Williford association  schemes to the Hollmann-Xiang fusion schemes. In this paper, we provide the answer by proving the following result:

\begin{maintheorem*}\label{main_th}
The  Penttila-Williford 3-class association  schemes and  the  Hollmann-Xiang  fusion association schemes  are isomorphic.
\end{maintheorem*}

The proof  essentially uses geometric arguments.  We start off with an explicit description of  the Penttila-Williford relative hemisystems in terms of coordinates in the projective space $\PG(3,q^2)$.  Via the Klein correspondence from the lines of $\PG(3,q^2)$ to the points of the Klein  quadric of $\PG(5,q^2)$, we obtain a geometric representation of the Penttila-Williford association schemes  in the orthogonal polar space $Q^-(5,q)$ whose points are the image of the  lines in  $H(3,q^2)$ \cite{hir}.  Thanks to	this representation we are able to find a desired isomorphism.

In \cite{hx} it was pointed out that a further fusion scheme of the 3-class Hollmann-Xiang scheme   produces a  strongly regular graph  with parameters $v=q^2(q^2-1)/2$, $k=(q^2+1)(q-1)$, $\lambda=q^2+q-2$, $\nu=2(q^2-q)$. These graphs have the same parameters of the ones found by R. Metz \cite{dt}, which can be also constructed from  a fusion of the Penttila-Williford schemes; see also \cite[p.189]{br}. These graphs are denoted by $NO^-(5,q)$ in Brouwer's table of strongly regular graphs \cite{brouwer}.

The paper \cite{hx} announced an alleged isomorphism between the above graphs in a forthcoming paper.  To the best of our knowledge, such a paper appears to have never been published. Anyway, the Main Theorem confirms the conjectured isomorphism.

The paper is structured as follows: in Section 2 we recall the construction of the Hollmann-Xiang  and  Penttila-Williford association schemes. In Sections 3 we give a coordinatization of the relative hemisystems of Penttila and Williford together with their representation in $Q^-(5,q)$. Finally, Section 4 contains the proof of the Main Theorem.

\section{Preliminaries} \label{sec_2}

For any given $n$-dimensional vector space $V=V(n,F)$ over a field $F$,  the {\em projective geometry defined by} $V$ is the partially ordered set of all subspaces of $V$, and it will be denoted by  $\PG(V)$. If $F$ is the finite field $\Fq$ with $q$ elements, then we may write $V=V(n,q)$ and $\PG(n-1,q)$ instead of $\PG(V)$. The 1-dimensional subspaces are called {\em points}, the 2-dimensional subspaces are called {\em lines}, and  the $(n-1)$-dimensional subspaces are called {\em hyperplanes} of $\PG(V)$. For a nonzero $v\in V$, $\<v\>$ will denote the point of $\PG(V)$ spanned by $v$. In order to simplify notation, for each subspace $U$ of $V$, that is an element of $\PG(V)$, we will use the same letter for the projective geometry defined by $U$.
If $V$ is endowed with a non-degenerate alternating, quadratic or hermitian form of Witt index $m$, the set $\cP$ of totally isotropic (or singular, in the case of quadratic form) subspaces of $V$ is a {\em polar space of rank $m$} of $\PG(V)$, which is called {\em symplectic, orthogonal}  or {\em unitary}, respectively. Our principal reference on projective geometries and polar spaces  is \cite{taylor}.

\subsection{The Hollmann-Xiang association schemes}

 A {\em non-degenerate conic} $\cC$ of $\PG(2,q^2)$ is an orthogonal polar space (of rank 1)  arising from a non-degenerate quadratic form $Q$ on  $V(3,q^2)$. A line $\ell$ of $\PG(2,q^2)$ is called a  {\em passant}, {\em tangent} or {\em secant}  of $\cC$ according as  $|\ell\cap\cC|=0$, 1 or 2.

Embed $\PG(2,q^2)$ in $\PG(2,q^4)$. Concretely this can be done by extending the scalars in  $V(3,q^2)$. It follows that each point of $\PG(2,q^2)$ extends to a point of $\PG(2,q^4)$. Similarly,   
each  line $\ell$ of $\PG(2,q^2)$  extends to a  line $\bar\ell$ of $\PG(2,q^4)$.
  The extension $\overline Q$ of  $Q$  in $V(3,q^4)$ is a non-degenerate quadratic form, and it defines a (non-degenerate) conic $\overline\cC$ in $\PG(2,q^4)$. While the extension $\bar\ell$ of a tangent  (or  secant) line $\ell$ of $\cC$ is  a tangent (or secant)  of $\overline\cC$,
  the extension of a passant line of $\cC$ is a secant of $\overline\cC$. Such a line  is called an {\em elliptic} line of $\overline\cC$, and we will denote by  $\cE$ the set of these lines.  Note that $\cE$ has size $(q^4-q^2)/2$.

Since all non-degenerate quadratic forms on $V(3,q^2)$ are equivalent, we may assume 
\[
\begin{array}{cccc}
\overline Q: & V(3,q^4) & \rightarrow & \Fqf \\
             & (x, y, z)  & \mapsto & y^2 - xz.
 \end{array}
 \]
Therefore,
 \[
\overline\cC =\{\<(1,t,t^2)\>:t\in\ \Fqf\}\cup \{\<(0,0,1)\>\}
\] and  
\[
\cC =\{\<(1,t,t^2)\>:t\in\ \Fqt\}\cup \{\<(0,0,1)\>\}.
\]
Therefore, for every elliptic line $\bar\ell$ of $\overline\cC$ we have  $\bar\ell\cap\overline\cC=\{\<(1,t,t^2)\>,\<(1,t^{q^2},t^{2q^2})\>\}$, for some  $t\in\Fqf\setminus\Fqt$. The reader is referred to \cite{hx} for more details.

Under the identification of $\Fqf\cup\{\infty\}$ with $\overline\cC$ given by 
\begin{equation}\label{eq_1}
\xi:  \ \   t   \leftrightarrow    \<(1,t,t^2)\>, \ \ \ 
    \   \infty    \leftrightarrow   \<(0,0,1)\>,
\end{equation}
 the pair $\t=\{t,t^{q^2}\}$, with $t\in\Fqf\setminus\Fqt$, may be associated with  the elliptic line intersecting $\overline\cC$ at  $\{\<(1,t,t^2)\>,\<(1,t^{q^2},t^{2q^2})\>\}$. We will use  $\bar\ell_{\t}$ to denote this line. 

We assume $q$ is even.  For any given pair of distinct elliptic lines $\bar\ell_{\s},\bar\ell_{\t}$, let
\begin{equation}\label{eq_2}
\widehat\rho(\bar\ell_{\s},\bar\ell_{\t})=\widehat\rho(\s,\t)=\frac{1}{\rho(s,t)+\rho(s,t)^{-1}},
\end{equation}
where
\begin{equation}\label{eq_10}
\rho(s,t)=\frac{(s+t)(s^{q^2}+t^{q^2})}{(s+t^{q^2})(s^{q^2}+t)}.
\end{equation}
It is evident that  $\Im\,\widehat\rho$ is a subset of $\Fqt$. The following result is straigtforward.
\begin{lemma}\em{\cite[Lemma 5.1]{hx}}\label{lem_1}
\[
\widehat\rho(\s,\t)=\frac{(s+t)(s^{q^2}+t^{q^2})(s+t^{q^2})(s^{q^2}+t)}{(s+s^{q^2})^2(t+t^{q^2})^2}=\left(\frac{1}{\rho(s,t)+1}\right)^2+\left(\frac{1}{\rho(s,t)+1}\right).
\]
\end{lemma}

%

Set $q=2^h$. For $r\in\{1,2\}$, 
let $\T_0(q^r)$ be the set of elements of $\Fqr$ with absolute trace zero:
\[
\T_0(q^r)=\left\{x \in \Fqr\,:\,\sum_{i=0}^{rh-1}{x^{2^i}} = 0\right\}.
\]
 In \cite{hx} Hollmann and Xiang consider the following sets to construct a 3-class association scheme:  
\[
\T_0=\T_0(q^2), \ \ \ \ \S_0^*=\T_0(q)\setminus\{0\},\ \ \ \ \S_1=\Fq\setminus\S_0.
\]
  Note that $\T_0=\{\alpha+\alpha^2:\alpha\in\Fqt\}$ as $q$ is even. By Lemma \ref{lem_1}, $\Im\,\widehat\rho$ is contained in $\T_0$.

%
\begin{theorem}{\em \cite{hx}}\label{th_1} On the set of the elliptic lines $\cE$ define the following relations:
\begin{itemize}
\item[$R_1$:] $(\bar\ell_{\s},\bar\ell_{\t})\in R_1$ if and only $\widehat\rho(\bar\ell_{\s},\bar\ell_{\t})\in\S_0^*$;
\item[$R_2$:]  $(\bar\ell_{\s},\bar\ell_{\t})\in R_2$ if and only  $\widehat\rho(\bar\ell_{\s},\bar\ell_{\t})\in\S_1$;
\item[$R_3$:] $(\bar\ell_{\s},\bar\ell_{\t})\in R_3$ if and only  $\widehat\rho(\bar\ell_{\s},\bar\ell_{\t})\in\T_0\setminus\Fq$.
\end{itemize}
 Then the pair $(\cE,\{R_i\}_{i=0}^{3})$, where $R_0$ is the identity relation, is a 3-class association scheme. 
\end{theorem}
The first eigen-matrix of the scheme is
\begin{equation}\label{eq_3}
P=\begin{pmatrix}
1 & (q-2)(q^2+1)/2 & q(q^2+1)/2 & q(q-2)(q^2+1)/2\\[.05in]
1 & -(q-1)(q-2)/2 & -q(q-1)/2 & q(q-2)\\[.05in]
1 & -(q^2-q+2)/2 & q(q+1)/2 & -q\\[.05in]
1 & q-1 & 0 & -q
\end{pmatrix};
\end{equation}
see \cite[Section 7]{hx}.

\begin{remark}\label{rem_1}
{\em By identification (\ref{eq_1}), the set $\cE$ may be replaced by the set $\cX=\{\t=\{t,t^{q^2}\}:t \in\Fqf\setminus\Fqt\}$ and the relations $R_i$, $i=1,2,3$, replaced by
\begin{itemize}
\item[$R_1'$:] $(\s,\t)\in R_1'$ if and only $\widehat\rho(\s,\t)\in\S_0^*$;
\item[$R_2'$:]  $(\s,\t)\in R_2'$ if and only  $\widehat\rho(\s,\t)\in\S_1$;
\item[$R_3'$:] $(\s,\t)\in R_3'$ if and only  $\widehat\rho(\s,\t)\in\T_0\setminus\Fq$;
\end{itemize}
here $\widehat\rho(\s,\t)$ is the quantity defined in (\ref{eq_2}). 
Hence, $(\cX,\{R_i'\}_{i=0}^{3})$ is an  association scheme isomorphic to  $(\cE,\{R_i\}_{i=0}^{3})$.}
\end{remark}

\begin{remark}\label{rem_2}
\emph{Actually, the scheme $(\cX,\{R_i'\}_{i=0}^{3})$ arises as a fusion of the one given by  the following result \cite{hx}.}

{\bf Theorem.}  Under the identification $\xi$, the action of $\PGL(2,q^2)$ on $\cE\times\cE$ gives rise to an association scheme on $\cX$ with $q^2/2-1$ classes $R_{\{\lambda,\lambda^{-1}\}}$, $\lambda\in\Fqt\setminus\{0,1\}$, where 
$(\s,\t)\in R_{\{\lambda,\lambda^{-1}\}}$ if and only if $\{\rho(s,t),\rho(s,t)^{-1}\}=\{\lambda,\lambda^{-1}\}$.
\end{remark}

%

\subsection{The Penttila-Williford association schemes}\label{sec_3}

Up to isometries, the vector space $V(4,q^2)$ has precisely one non-degenerate hermitian form, and its Witt index is 2. As usual,  $H(3,q^2)$ denotes the  unitary polar space of rank 2 defined by it. A {\em point} (resp. {\em line}) of $H(3,q^2)$ is a 1-dimensional (resp. 2-dimensional)  subspace in $H(3,q^2)$.

Assume $q$ even for the rest of the current section.
In  $V(4,q^2)$ there is  a 4-dimensional $\Fq$-vector space $\widehat V$ such that the  restriction of the hermitian form on it induces a non-degenerate alternating form $\widehat b$ which defines a symplectic polar space $W(3,q)$ of rank 2 of $\PG(\widehat V)$ \cite{ck}. In addition,  $\widehat b$ is  the polar of a non-degenerate quadratic form $\widehat Q$ of Witt index 1, whose set of singular point is denoted by $Q^-(3,q)$. 
By $\widehat\cW$ (resp. $\widehat\cQ$) we denote the set of all the totally isotropic (resp. singular) subspaces of $W(3,q)$ (resp. $Q^-(3,q)$)  extended over $\Fqt$.
As a consequence,  for every point of $H(3,q^2)$ not in $\widehat\cW$ there are exactly $q$ lines of $H(3,q^2)$  disjoint from  $\widehat\cW$ and  one in $\widehat \cW$.  Note that $\widehat\cW$ is an embedding of $W(3,q)$ in $H(3,q^2)$.

The following definition was introduced in \cite{pw}. A {\em relative hemisystem of $H(3,q^2)$ with respect to} $W(3,q)$ is a set $\cH$ of lines of $H(3,q^2)$ disjoint from $\widehat\cW$ such that every point of $H(3,q^2)$ not in  $\widehat\cW$ lies on exactly $q/2$ lines of $\cH$.
For any given line $l$ of $H(3,q^2)$ disjoint from $\widehat\cW$, let $\cS_l$ denote the set of lines of $H(3,q^2)$ which meet both $l$ and $\widehat\cW$. We stress the fact that  $\cS_l$ consists of the lines of $\widehat \cW$ that extend  elements of a regular spread of $W(3,q)$\footnote{A spread of $W(3,q)$ in $\PG(3,q)$ is a set $\cS$ of totally isotropic lines which partition the pointset of $\PG(3,q)$.
 $\cS$ is {\em regular} if for any three distinct lines of $\cS$ there is a set $R$ of $q +1$ lines of $\cS$ containing them, with the following property: any line of $\PG(3,q)$ intersecting three lines in $R$ meets all the lines of $R$.}, and refer to $\cS_l$ as the {\em spread  subtended by} $l$.
 

\begin{theorem}{\em \cite[Theorem 4]{pw}} \label{th_2} Let $\cH$ be a relative hemisystem  of $H(3,q^2)$ with respect to $W(3,q)$.  Then a primitive $Q$-polynomial 3-class   association scheme can be constructed on $\cH$ by the defining the following relations: 
\begin{itemize}
\item[$\widetilde R_1$:] $(l,m)\in \widetilde R_1$ if and only $|l\cap m|=1$;
\item[$\widetilde R_2$:]  $(l,m)\in \widetilde R_2$ if and only $l\cap m=\emptyset$ and $|\cS_l\cap\cS_m|=1$;
\item[$\widetilde R_3$:] $(l,m)\in \widetilde R_3$ if and only $l\cap m=\emptyset$ are $|\cS_l\cap\cS_m|=q+1$.
\end{itemize}
\end{theorem}
Let $\PO^-(\widehat V)$ be the stabilizer of $\widehat\cQ$  in the projective  unitary group $\PGU(4,q^2)$. By looking at the action of the commutator subgroup $\mP\Omega^-(\widehat V)$ of $\PO^-(\widehat V)$ on the lines of $H(3,q^2)$,   the following result was proved in \cite{pw}.
\begin{theorem}\label{th_3}
 $\mP\Omega^-(\widehat V)$ has two orbits on the lines of $H(3,q^2)$ disjoint from $\widehat\cW$, and each  orbit    is a relative hemisystem with respect to $W(3,q)$.
\end{theorem}

We consider an association scheme $(\cH,\{\widetilde R_i\}_{i=0}^{3})$ as in Theorem \ref{th_2} by using the hemisystems from Theorem \ref{th_3}. As expected, the first eigen-matrix of the scheme  is precisely the matrix in (\ref{eq_3}). 

\section{The explicit construction of the relative hemisystem of Penttila-Williford}\label{sec_4}

Let $G$ and $H$ be groups acting on the sets $\Omega$ and $\Delta$, respectively. The two actions are said to be {\em permutationally isomorphic} if there exist a bijection $\theta: \Omega \rightarrow \Delta$ and an isomorphism $\chi: G \rightarrow H$ such that the following diagram commutes:

\begin{center}
\begin{tikzpicture}

\node (G) at (0, 1.9) {$G$};
\node (X) at (0.39, 1.9) {$\times$};
\node (Omega1) at (.78, 1.9) {$\Omega$};
\node (Omega2) at (3.3, 1.9) {$\Omega$};

\node (H) at (0, 0) {$H$};
\node (X) at (0.39, 0) {$\times$};
\node (Delta1) at (.78, 0) {$\Delta$};
\node (Delta2) at (3.3, 0) {$\Delta$};
\draw[->]   (Omega1) edge node [above] {$\phi$}  (Omega2) ;
\draw[->]   (Delta1) edge node [above] {$\tilde\phi$}  (Delta2) ;
\draw[->]   (G) edge node [left] {$\chi$}  (H) ;
\draw[->]   (Omega1) edge node [left] {$\theta$}  (Delta1) ;
\draw[->]   (Omega2) edge node [left] {$\theta$}  (Delta2) ;
\end{tikzpicture}
\end{center}
Here $\phi$ and $\tilde\phi$ are the maps defining the action of $G$ and  $H$ on $\Omega$ and $\Delta$, respectively.

Let $Q^-(3,q)$ be the orthogonal polar space  (of rank 1) defined by $\widehat Q$ on the  4-dimensional $\Fq$-vector space $\widehat V$ introduced in Section \ref{sec_3}. \\
It is known that $(\PSL(2,q^2),\PG(1,q^2))$ and $(\mathrm P\Omega^-(\widehat V),Q^-(3,q))$ are  permutationally isomorphic   for all prime powers $q$. For sake of completeness, we give an explicit description of the above isomorphism which is more suitable for our computation.

\comment{The vector space $V(2n,q)$ has precisely  two  (non-degenerate) quadratic forms, and they differ by their Witt-index, that is the dimension of their maximal totally singular subspaces; see \cite{taylor}. These dimensions are $n-1$ and $n$, and the quadratic form is {\em elliptic} or {\em hyperbolic}, respectively. In terms of the associated projective geometry $\PG(2n - 1, q)$, the kernel of the elliptic (resp. hyperbolic) quadratic form defines an {\em elliptic} (resp. {\em hyperbolic}) quadric of $\PG(2n-1, q)$. The vector space $V(2n+1,q)$, $q$ even,  has precisely  one  (non-degenerate) quadratic form, and its Witt index is $n$.  The quadratic form is called {\em parabolic}. In terms of the associated projective geometry $\PG(2n, q)$, its kernel defines a {\em parabolic} quadric of $\PG(2n, q)$. 
}

In $V(4,q^2)=\{(X_1,X_2,X_3,X_4): X_i\in\Fqt\}$, let $\widehat{V}$ be the set of all vectors $v=(\alpha,x^q,x,\beta)$ with $\alpha,\beta \in\Fq$, $x\in\Fqt$. With the usual sum and  multiplication by scalars from $\Fq$, $\widehat{V}$ is a 4-dimensional vector space over $\Fq$.

As usual we  identify $\PG(1,q^2)$ with $\Fqt\cup\{\infty\}$ and we consider the following injective map:
\begin{equation}\label{eq_7}
\begin{array}{rcccc}
\theta: & \Fqt\cup\{\infty\} & \longrightarrow  & \PG(\widehat V)\\[.02in]
        &  t & \mapsto     & \<(1,t^q,t,t^{q+1})\>\\[.02in]
        &  \infty & \mapsto     & \<(0,0,0,1)\>
\end{array}\ \ \ .     
\end{equation}

\begin{proposition}{\em \cite{cs}}
The image of $\theta$  is  an  orthogonal polar space  of rank 1 of  $\PG(\widehat V)$.
\end{proposition}
\begin{proof}
Let $Q$ be the quadratic form on $V(4,q^2)$ defined by
\[
Q(\X)=X_1X_4-X_2X_3,
\]
which has $b(\X,\Y)=X_1Y_4 +X_4Y_1 -X_2Y_3-X_3Y_2$ as the associated non-degenerate bilinear  form. The restriction $\widehat Q=Q|_{\widehat V}$ is the quadratic form given by
\[
\widehat Q(v)=\alpha\beta-x^{q+1},
\]
which has 
\begin{equation}\label{eq_12}
\widehat b(v,v')=\alpha\beta'+\beta\alpha'-x x'^q-x^qx'
\end{equation}
as the associated bilinear form. Let $v=(\alpha,x^q,x,\beta)\in\Rad(\widehat V)$, that is $\widehat b(v,v')=0$, for all $v'\in\widehat V$. If $\alpha'=\beta'=0$, a necessary condition for $v\in\Rad(\widehat V)$ is
\[
x^qx'+x x'^q=0,
\]
for all $x'\in\Fqt$. This shows that the polynomial in $x'$ of degree $q$ on the left hand side has at least $q^2$ roots. Therefore, it  must be the zero polynomial, and $x=0$. We repeat the above argument for $\alpha'=x'=0$ and  for $x'=\beta'=0$ to show that $v=0$. This yields that $\widehat b$, and hence $\widehat Q$, is non-degenerate.  Let $u$ be a singular vector for $\widehat Q$. Without loss of generality we may take $u=(1,0,0,0)$. Therefore,  the subspace $U=\{v\in\widehat V:\widehat b(v,u)=0\}$ coincides with $\{(\alpha,x^q,x,0):\beta \in\Fq,x\in\Fqt\}$. It is easily seen that $U\cap \ker \widehat Q=\{\alpha u:\alpha \in \Fq\}$. Thus, $\widehat Q$ is a quadratic form of Witt index 1 giving rise to the orthogonal polar space 
\[
Q^-(3,q)=\{\<(1,t^q,t,t^{q+1})\>:t\in\Fqt\}\cup \{\<(0,0,0,1)\>\}, 
\] 
which is precisely $\Im\,\theta$.
\end{proof}
Let $\chi$ be the monomorphism defined by
\[
\begin{array}{rcccc}
\chi: & \SL(2,q^2)& \longrightarrow & \SL(4,q^2) \\
      &    g &\mapsto & g\otimes g^q
\end{array},
\]
where $\otimes$ is the Kronecker product and  $g^q$ denotes the matrix $g$ with  its entries raised to the $q$-th power. It is straightforward to check that $\chi(g)$ is a $\widehat Q$-isometry, for every $g\in\SL(2,q^2)$. 
 Therefore, $\chi$ can be regarded as  a monomorphism from $\PSL(2,q^2)$ to  $\PO^-(\widehat V)$. It is actually an isomorphism from $\PSL(2,q^2)$ to  $\mP\Omega^-(\widehat V)$, as it will be shown below.

Let $t_a$ be the transvection in $\SL(2,q^2)$ with  matrix  $\begin{pmatrix} \ 1 & 0\ \\ \ a & 1 \ \end{pmatrix}$, for some $a\in\Fqt^*$. The isometry $\chi(t_a)$ maps  $(\alpha,x^q,x,\beta)$ to $(\alpha,a^q\alpha+x^q,a\alpha + x,a^{q+1}\alpha+ax^q+a^qx+\beta)$. Its restriction on the hyperplane of all  $\widehat b$-orthogonal vectors to $u=(0,0,0,1)$   is the map
\[
\eta_{u,y}(w)=w+\widehat b(w,y)u,
\]
where $y=(0,-a^q,-a,0)$. This yields that $\chi(t_a)$  is actually the unique Siegel transformation $\rho_{u,y}$ which extends $\eta_{u,y}$ \cite[Theorem 11.18]{taylor}. By using \cite[Theorem 11.19 (ii)]{taylor} it is possible to show that as $a$ varies in $\Fqt^*$,  $\rho_{u,y}$  describes all the Siegel transformations centered at $u$. 

Every transvection $g$ is conjugate in $\SL(2,q^2)$ to a transvection of type $t_a$. This implies that $\chi(g)$ is also a Siegel transformation  \cite[Theorem 11.19 (iii)]{taylor}. Therefore, $\chi$ gives rise to a  bijection from the set of all transvections in $\SL(2,q^2)$  to all Siegel transformations of $\widehat V$. Since transvections generate $\SL(2,q^2)$, and Siegel transformations generate $\Omega^-(\widehat V)$, we achieve $\chi(\PSL(2,q^2))\le\mP\Omega^-(\widehat V)$. As  $|\PSL(2,q^2)|=|\mP\Omega^-(\widehat V)|$, $\chi$ is actually the desired  isomorphism.  It is a matter of fact that the diagram

\begin{center}
\begin{tikzpicture}

\node (G) at (0.1, 1.9) {$\PSL(2,q^2)$};
\node (X) at (1.3, 1.9) {$\times$};
\node (Omega1) at (2.47, 1.9) {$\Fqt\cup\{\infty\}$};
\node (Omega2) at (6.3, 1.9) {$\Fqt\cup\{\infty\}$};
\node (H) at (0.1, 0) {$\mP\Omega^-(\widehat V)$};
\node (X) at (1.29, 0) {$\times$};
\node (Delta1) at (2.47, 0) {$Q^-(3,q)$};
\node (Delta2) at (6.3, 0) {$Q^-(3,q)$};
\draw[->]   (Omega1) edge node [above] {$\phi$}  (Omega2) ;
\draw[->]   (Delta1) edge node [above] {$\tilde\phi$}  (Delta2) ;
\draw[->]   (G) edge node [left] {$\chi$}  (H) ;
\draw[->]   (Omega1) edge node [left] {$\theta$}  (Delta1) ;
\draw[->]   (Omega2) edge node [left] {$\theta$}  (Delta2) ;
\end{tikzpicture}
\end{center}
commutes.

For the rest of this section, assume $q$ is even.
The bilinear form $\widehat b$ defined by (\ref{eq_12}) is a (non-degenerate) alternating form on $\widehat V$.  
Let $h$ be  the  non-degenerate hermitian form on $V(4,q^2)$ given by
\[
h(\X,\Y)=X_1Y_4^q+X_2Y_2^q+X_3 Y_3^q+X_4 Y_1^q, 
\]
with   associated  unitary polar space $H(3,q^2)$.  It is evident that $h|_{\widehat V}=\widehat b$. Therefore, the  symplectic polar space $W(3,q)$ defined by $\widehat b$, as well as the orthogonal polar space $Q^-(3,q)$, can be embedded  in $H(3,q^2)$ by extending the scalars, so getting $\widehat\cW$ and $\widehat\cQ$ introduced in Section \ref{sec_3}. This also implies that   $\mP\Omega^-(\widehat V)$ is a subgroup of the projective symplectic group $\PSp(\widehat V)$ which is in turn a subgroup of the projective unitary group $\PGU(4,q^2)$.

The semilinear involutorial transformation $\tau$ of $V(4,q^2)$ given by
\[
\begin{array}{rccc}
\tau: & V(4,q^2) & \longrightarrow  & V(4,q^2) \\
      & (X_1,X_2,X_3,X_4) &\mapsto  & (X_1^q,X_3^q,X_2^q,X_4^q).
\end{array}
\]
fixes $H(3,q^2)$ and  acts as the identity on $\widehat\cW$. 

If we embed  $V(4,q^2)$ in $V(4,q^4)$ by extending the scalars, then $\PG(3,q^4)$  embeds  $\PG(\widehat V)$. Therefore,  $\theta$ defined by (\ref{eq_7}) can be naturally thought as the restriction of the following map:
\[
\begin{array}{ccccccc}
\theta:  & \Fqf\cup\{\infty\} & \longrightarrow   & \PG(3,q^4) \\[.1in]
       & t                  & \longmapsto    & \<(1,t^q,t,t^{q+1})\>\\[.05in]
       & \infty             & \longmapsto    &  \<(0,0,0,1)\>.
\end{array}
\]

%
%
%
Note that, for any $t\in\Fqf\setminus\Fqt$,  $\theta(t)$ is not the span of a vector of $V(4,q^2)$. Moreover,   
\[
\theta(t^{q^2})=\<(1,t^{q^3},t^{q^2},t^{q^3+q^2})\>=\theta(t)^{\tau^2}\neq\theta(t).
\]
%
For each $t\in\Fqf\setminus\Fqt$, we  associate the pair $\t=\{t,t^{q^2}\}$  with the line $M_{\t}$ of $\PG(3,q^4)$ spanned by $\theta(t)$ and $\theta(t^{q^2})$, which is distinct from $M_{\t}^\tau$.
\begin{lemma}\label{lem_2}
For each pair $\t$,  $M_{\t}\cap V(4,q^2)$ is a line of $H(3,q^2)$, say $m_\t$, which is disjoint from $\widehat\cW$.
\end{lemma}
\begin{proof}
A straightforward computation shows that the vectors  in $M_{\t}\cap V(4,q^2)$ are precisely  $\X_\lambda=(\lambda+\lambda^{q^2}, \lambda 
t^q+\lambda^{q^2} t^{q^3},\lambda t+\lambda^{q^2} t^{q^2},\lambda t^{q+1}+\lambda^{q^2} t^{q^3+q^2})$, for all $\lambda\in\Fqf$, and they form a line $m_\t$ of $\PG(3,q^2)$. Since $h(\X_{\lambda}, \X_{\lambda})=0$ for all $\lambda\in\Fqf$,  $m_\t$ is a line of $H(3,q^2)$.  Finally, in order to prove that $m_{\t}$ is disjoint from $\widehat\cW$, consider the following system:
\begin{equation}\label{eq_6}
\begin{aligned}
&a \,\alpha  \ =\lambda+\lambda^{q^2}\\
&a \, x^{q} =\lambda t^{q} +\lambda^{q^2} t^{q^3}\\
&a \, x \ =\lambda t +\lambda^{q^2} t^{q^2}\\
&a \, \beta \ =\lambda t^{q+1}+\lambda^{q^2} t^{q^3+q^2}\ , 
\end{aligned}
\medskip
\end{equation}
with $\alpha, \beta \in \Fq$, $x,a \in \Fqt$, $\lambda \in \Fqf$, $t \in \Fqf \setminus \Fqt$. The existence of a solution for \eqref{eq_6}, or rather the existence of $a \in \Fqt$,  makes the system  inconsistent.
This concludes the proof.
\end{proof}  
\begin{proposition}
The sets $\{m_{\t}:t \in\Fqf\setminus \Fqt\}$ and  $\{m_{\t}^\tau:t \in\Fqf\setminus \Fqt\}$ are precisely the two orbits of $\mP\Omega^-(\widehat V)$ on the lines of $H(3,q^2)$  disjoint from $\widehat\cW$.
\end{proposition}
\begin{proof}
From the proof of  \cite[Theorem 5]{pw}, $\mP\Omega^-(\widehat V)$ has two orbits on the lines of $H(3,q^2)$  disjoint from $\widehat\cW$, and  these two orbits are interchanged by $\tau$. We recall that $m_{\t}$ is uniquely defined by the line $M_{\t}$ of $\PG(3,q^4)$, which is spanned by $\theta(t)$ and $\theta(t^{q^2})$. Hence, it suffices  to prove that $\{M_{\t}:t \in\Fqf\setminus \Fqt\}$ is an   orbit of $\mP\Omega^-(\widehat V)$.


Let $\omega \in \Fqf \setminus \Fqt$ such that $\omega^{q^2}=\omega+1$ and  $\omega^2+\omega=\delta$, with $\delta \in \Fqt \setminus \T_0$, $\delta\neq 1$.
For all $t \in \Fqf \setminus \Fqt$, write  $t=x+y \omega$, with $x, y \in \Fqt$, $y \neq 0$. 

As a group acting on the projective line $\PG(1,q^2)$ assimilated to the set $\Fqt\cup\{\infty\}$,  $\PSL(2,q^2)$ may be identified with the group of linear fractional transformations
\[
z\mapsto \frac{az+b}{cz+d},
\]
where $ad-bc$ is a non-zero square in $\Fqt$ \cite{taylor}.  For any given  $t=x+y \omega \in \Fqf \setminus \Fqt$, let $ g\in\PSL(2,q^2)$ with matrix  $[\,1,0;\,x,y\,]$. Then, $\chi(g)=g\otimes g^q$ maps $\left(\theta(\omega),\theta(\omega^{q^2})\right)$ to $\left(\theta(t),\theta(t^{q^2})\right)$  by taking into account $\omega^{q^2}=\omega+1$. This implies that $\chi(g)\in\mP\Omega(\widehat V)$  maps the  line $M_{\{\omega, \omega^{q^2}\}}$  to $M_{\t}$. 
\end{proof}

\begin{corollary}\label{cor_1}
The sets $\{m_{\t}:t\in\Fqf\setminus\Fqt\}$ and $\{m_{\t}^\tau:t\in\Fqf\setminus\Fqt\}$ are the Penttila-Williford relative hemisystems.
\end{corollary}

\section{The proof of the Main Theorem}

Define $T=\{\{t,t^{q^2}\}: t \in \Fqf\setminus\Fqt\}$, and put $\cX=T$.
By Remark \ref{rem_1}, we need to find a bijection between  the set $\cX$  and the relative hemisystem $\cH=\{m_{\t}:\t\in T \}$ preserving the relations defined on them. 

From the arguments in Section \ref{sec_4}, we may associate the pair $\t=\{t,t^{q^2}\}\in\cX$ with the line $m_\t\in\cH$. Moreover, Corollary \ref{cor_1} gives $|\cH|=(q^4-q^2)/2=|\cX|$, and this contributes to make the mapping $\varphi:\cX\to\cH, \t\mapsto m_{\t}$ a bijection. In order to show that $\varphi$, in fact, preserves the relations, we will move into a different geometric setting. More precisely, we will use the following  dual representation of  $H(3,q^2)$. Via the Klein correspondence $\kappa$, the lines of $\PG(3,q^2)$ are mapped to the points of an orthogonal polar space $Q^+(5,q^2)$ of rank 3 of $\PG(5,q^2)$, which is the so-called the {\em Klein quadric}.  In particular, the lines of  $H(3,q^2)$ are mapped to the points of an orthogonal polar space   $Q^-(5,q)$ of rank 2 in a $\PG(5,q)$   embedded in $\PG(5,q^2)$.  When $q$ is even, $\kappa$ maps the lines of any symplectic polar space of rank 2 embedded in $H(3,q^2)$ to the points of an orthogonal polar space of rank 2, which is the intersection of $Q^-(5,q)$ with a hyperplane of $\PG(5,q)$. The reader is referred to  \cite{hir} for more details on the  Klein correspondence. 
  
Assume $q$ even. In $V(6,q^2)$  consider the 6-dimensional $\Fq$-subspace $\widetilde{V}=\{(x,x^q,y,y^q, z,z^q):x,y,z\in\Fqt\}$.  Let $\PG(\widetilde V)$ be the projective geometry defined by $\widetilde V$. 

We consider the Klein quadric $Q^+(5,q^2)$  defined by the (non-degenerate) quadratic form
$
Q(\X)=X_1X_6+X_2X_5+X_3X_4$
on $V(6,q^2)$. For any given $w=(x,x^q,y,y^q, z,z^q)\in\widetilde V$,  
\[
\widetilde{Q}(w)=Q|_{\widetilde V}(w)=xz^q+x^qz+y^{q+1}.
\]
From \cite[Proposition 2.4]{ks},   $\widetilde{Q}$  is a non-degenerate quadratic form of Witt index 2 on $\widetilde V$ with associated alternating form 
\[
\widetilde b(w,w')=xz'^q+x^qz'+yy'^q+y^qy'+zx'^q+z^qx'.
\]
 Therefore, $\widetilde{Q}$  gives rise to an orthogonal polar space $Q^-(5,q)$ of $\PG(\widetilde{V})$ embedded in $Q^+(5,q^2)$.

For any  subspace  $X$ of $\widetilde{V}$, set 
\[
X^\perp=\{w \in \widetilde{V}:  \widetilde b(w,u)=0, \mathrm{\ for\ all\ }u\in X \}.
\]

Let  $Q(4,q)$ be the polar space whose points are the $\kappa$-image of the lines of $\widehat\cW$, and  $\Gamma$ be the hyperplane of $\PG(\widetilde V)$ containing $Q(4,q)$. For a complete description of $\Gamma$ observe that the pairs 
\[
\{(1,0,0,0),(0,x,x^q,0)\}, \ \ \{(0,0,0,1),(0,x,x^q,0)\}, \ \  \{(1,1,1,1),(x+x^q,x,x^q,0)\}, 
\]
 with $x\in\Fqt$, span lines  of  $\widehat\cW$ which give three skew lines of $Q(4,q)$ under $\kappa$ generating $\Gamma$. It follows that $\Gamma=\{(x,x^q,c,c,z,z^q):x,z\in\Fqt,c\in\Fq\}$.

Under $\kappa$, the line $m_{\t}$ of $\cH$  is mapped to the point $P_{\t}=\<w_{\t}\>$  of $\PG(\widetilde V)$, where
\[
w_{\t}=(t^q+t^{q^3}, t+t^{q^2},t^{1+q}+t^{q^2+q^3}, t^{1+q^3}+t^{q+q^2},t^{1+q+q^3}+t^{q+q^2+q^3}, t^{1+q+q^2}+t^{1+q^2+q^3}).
\]
Note that $P_{\t}$  is  in $Q^-(5,q)$, but not in  $Q(4,q)$. Let $P'_\t=\kappa(m^{\tau}_{\t})$.   Since $m_{\t}$ and $m^{\tau}_{\t}$ are disjoint lines of $H(3,q^2)$, the line $L_{\t}$ spanned by  $P_{\t}$ and $P'_{\t}$ intersects $Q^-(5,q)$ just at $P_{\t}$ and $P'_{\t}$. 
On the other hand, $m_\t$ and $m^{\tau}_{\t}$ subtend the same spread $\cS_\t=\cS_{m_\t}$ in $\widehat\cW$. The $\kappa$-image of $\cS_\t$ is an orthogonal polar space of rank 1 contained in $Q(4,q)$ \cite{hir}, and it  turns out this is precisely $Q(4,q)\cap L_{\t}^\perp$. Consequently,  $L_{\t}^{\perp}$ is  in $\Gamma$ and $\Gamma^{\perp}=\<(0,0,1,1,0,0)\>=\<w_0\>$ is a point of $L_{\t}$, for all $t\in\Fqf\setminus\Fqt$. The symbol  $\widetilde\cO_{\t}$ will be used to indicate  $Q(4,q)\cap L_{\t}^\perp$.

\comment{Via $\kappa$, the involution $\tau$ of $V(4,q^2)$ acting on the lines of $\PG(3,q^2)$ produces the involutorial semilinear transformation
\[
\begin{array}{rccc}
T: & V(6,q^2) & \longrightarrow  & V(6,q^2) \\
      & (X_1,X_2,X_3,X_4,X_5,X_6) &\mapsto  & (X_2^q,X_1^q,X_3^q,X_4^q,X_6^q,X_5^q)
\end{array},
\]
which gives rise to a projectivity  of $\PG(\widetilde V)$ that preserves $Q^-(5,q)$, fixes $Q(4,q)$ pointwise and interchanges $P_{\t}$ with $P'_\t$. Hence, $L_{\t}^{\perp}$ is  contained in $\Gamma$. This implies that $\Gamma^{\perp}=\<(0,0,1,1,0,0)\>$ is a point of $L_{\t}$, for all $t\in\Fqf\setminus\Fqt$.

The symbol  $\widetilde\cO_{\t}$ will be used to indicate  $Q(4,q)\cap L_{\t}^\perp$.   Note that $\widetilde\cO_{\t}$ is the $\kappa$-image of the regular spread $\cS_{m_{\t}}$, and it is an elliptic quadric in $L_{\t}^\perp$.
}

For any given distinct pairs $\s$ and $\t$, let $\Pi_{\s,\t}$ be the plane of $\PG(\widetilde V)$ spanned by $\Gamma^\perp$, $P_{\s}$ and $P_{\t}$.
The restriction  of $\widetilde Q$ and  $\widetilde b$ on $\Pi_{\s,\t}$ will be denoted by $\widetilde Q_{\s,\t}$ and $\widetilde b_{\s,\t}$, respectively. Identifying a triple $(a,b,c)\in\Fq^3$ with the vector $v=aw_\s+b w_0+c w_\t \in \Pi_{\s,\t}$, we obtain that the action of $\widetilde b_{\s,\t}$ induced on $\Fq^3$ is given by the matrix
\[
B=\begin{pmatrix}
\widetilde b(w_{\s},w_{\s}) & \widetilde b(w_{\s},w_{0}) & \widetilde b(w_{\s},w_{\t}) \\
\widetilde b(w_{0},w_{\s}) & \widetilde b(w_{0},w_{0}) & \widetilde b(w_0,w_{\t})   \\
\widetilde b(w_{\t},w_{\s}) & \widetilde b(w_{\t},w_{0}) & \widetilde b(w_{\t},w_{\t}) 
\end{pmatrix}=\begin{pmatrix}
0 & \Tr(s^{q+1}) & \widetilde b(w_{\s},w_{\t}) \\
\Tr(s^{q+1}) & 0 & \Tr(t^{q+1})  \\
\widetilde b(w_{\s},w_{\t}) & \Tr(t^{q+1})  & 0
\end{pmatrix};
\]
here $\Tr$ is the trace map from $\Fqf$ on $\Fq$. A straightforward calculation shows that $\Pi_{\s,\t}$ is degenerate as $\Rad(\Pi_{\s,\t})=\<v_{\s,\t}\>$, where 
\[
v_{\s,\t}=
\Tr(t^{q+1})w_\s+ \widetilde b(w_{\s},w_{\t})w_0+ \Tr(s^{q+1})w_\t.
\]
It is easily seen that
\begin{equation}\label{eq_9}
\begin{aligned}
\widetilde Q_{\s,\t}(v_{\s,\t})&=\widetilde b(w_{\s},w_{\t})\left(\widetilde b(w_{\s},w_{\t})+\Tr(s^{q+1})\Tr(t^{q+1})\right)\\
&=\widetilde b(w_{\s},w_{\t}) \ \widetilde b(w_{\s},w'_{\t}),
\end{aligned}
\end{equation}
where 
\[
w'_\t= (t^q+t^{q^3}, t+t^{q^2},t^{q+q^2}+t^{1+q^3}, t^{1+q}+t^{q^2+q^3},t^{1+q+q^3}+t^{q+q^2+q^3}, t^{1+q+q^2}+t^{1+q^2+q^3}).
\]
Note that $P'_\t=\kappa(m_\t^\tau)=\<w'_\t\>$. 

Now two cases are possible according as  $v_{\s,\t}$ is singular or not. 

If $\widetilde Q_{\s,\t}(v_{\s,\t})=0$, then $\widetilde Q_{\s,\t}$ is degenerate, and $\cC_{\s,\t}=\Pi_{\s,\t}\cap Q^-(5,q)$ consists of two distinct lines through $\<v_{\s,\t}\>$, as $P_{\s}$, $P'_{\s}$, $P_{\t}$ and  $P'_{\t}$ are distinct points no three of them collinear. This yields that $L_{\s}^\perp$ meets $L_{\t}^\perp$ in the plane  $\Pi_{\s,\t}^\perp$ of $\Gamma$, with $\widetilde Q|_{\Pi_{\s,\t}^\perp}$ degenerate, and $\widetilde\cO_{\s} \cap \widetilde\cO_{\t}=\Pi_{\s,\t}^\perp\cap Q^-(5,q)=\<v_{\s,\t}\>$. By taking into account \eqref{eq_9}, there are two possibilities of obtaining zero for $\widetilde Q_{\s,\t}(v_{\s,\t})$: either $\widetilde b(w_{\s},w_{\t})=0$ or $\widetilde b(w_{\s},w'_{\t})=0$.

If $\widetilde b(w_{\s},w_{\t})=0$, then $P_\s$ and $P_\t$ are collinear in $Q^-(5,q)$, or equivalently,  the lines $m_\s$ and $m_\t$ are concurrent, that is $(m_\s, m_\t) \in \widetilde R_1$ (see Theorem \ref{th_2}). On the other hand, by taking into account (\ref{eq_10}),
\[
\begin{aligned}
0=\widetilde b( & w_{\s}, w_{\t})\\ & =(s^{q^2}+s)^q(t^{q^2}+t)^q(s+t^{q^2})(s^{q^2}+t)+(s^{q^2}+s)(t^{q^2}+t)(s+t^{q^2})^q(s^{q^2}+t)^q
\end{aligned}
\]
if and only if  
\[
\nu=\frac{1}{\rho(s,t)+1}=\frac{(s+t^{q^2})(s^{q^2}+t)}{(s^{q^2}+s)(t^{q^2}+t)}\in \Fq.
\]

When $\nu\in\Fq$,  $\widehat\rho(s,t)\in\S_0^*$ by   Lemma \ref{lem_1}, that is $(\s,\t)\in R'_1$ (see Remark \ref{rem_1}).
On the other hand, if $\widehat\rho(s,t)=\nu^2+\nu\in\S_0^*$, then there exists $z\in\Fq$ such that $(z+\nu)^2+(z+\nu)=0$, which implies either $z=\nu$ or $z+1=\nu$. In both cases $\nu\in\Fq$. Therefore, $(m_\s, m_\t) \in \widetilde R_1$  if and only if $(\s,\t)\in R'_1$.

If $\widetilde b(w_{\s},w'_{\t})=0$, then  $P_\s$ and $P'_\t$ are collinear in $Q^-(5,q)$, and this  leads to the non-collinearity of $P_\s$  and  $P_\t$. This means that $(m_\s, m_\t) \in \widetilde R_2$ on one side, and   $(\s,\t)\in R'_2$ on the other one. In fact,
\[
\begin{aligned}
0=\widetilde b( & w_{\s}, w'_{\t})\\ =&(s^{q^2}+s)^q(t^{q^2}+t)^q(s+t^{q^2})(s^{q^2}+t)+(s^{q^2}+s)(t^{q^2}+t)(s+t^{q^2})^q(s^{q^2}+t)^q+\\ & +(s^{q^2}+s)^{q+1}(t^{q^2}+t)^{q+1}
\end{aligned}
\]
if and only if
\[
\nu^q+\nu=\Bigg(\frac{1}{\rho(s,t)+1}\Bigg)^q +\frac{1}{\rho(s,t)+1}=1.
\]
When $\nu^q+\nu=1$, then $\nu\notin\Fq$.  This implies that the equation $Z^2+Z=\widehat\rho(s,t)$ has no solutions in $\Fq$, that is $\widehat\rho(s,t)\in\S_1$, i.e. $(\s,\t)\in R'_2$. On the other hand,  $\widehat\rho(s,t)=\nu^2+\nu\in\S_1\subset\Fq$ implies $\nu\not\in\Fq$. As $\widehat\rho(s,t)\in\Fq$, then  $(\nu^q+\nu)^2+(\nu^q+\nu)=0$ holds, whence $\nu^q+\nu=1$.

Finally, if $\widetilde Q_{\s,\t}(v_{\s,\t})\neq0$, then $\widetilde Q_{\s,\t}$ is non-degenerate and $\<v_{\s,\t}\>$ is the nucleus of the (non-degenerate) conic $\cC_{\s,\t}$. Therefore, $\widetilde\cO_{\t}$ and $\widetilde\cO_{\s}$ meet in  $q+1$ points of $\Pi_{\s,\t}^\perp\cap Q(4,q)$. Then, $\cS_{m_\t}=\kappa^{-1}(\widetilde \cO_\t)$ and  $\cS_{m_\s}=\kappa^{-1}(\widetilde \cO_\s)$ meet in exactly $q+1$ lines in $\widehat\cW$, that is $(m_\t,m_\s)\in \widetilde R_3$. 
It is clear that $(m_\s,m_\t)\in\widetilde R_3$ if and only if $(\s, \t) \in R'_3$ by exclusion.

Summing up, for each $i=0,\ldots,3$, we have 
\begin{equation*}
(\s, \t) \in R'_i \mathit{\ \ \ \ \ if\ and\ only\ if\ \ \ \ \ } (m_\s, m_\t)=\varphi(\s, \t) \in \widetilde{R}_i,\\
\end{equation*}
i.e. $\varphi$ induces a bijection between $R'_i$ and $\widetilde{R}_i$, thus achieving our aim.
\section*{Acknowledgments}
\noindent The authors would like to thank the referees for their 
useful comments, and in particular for pointing us out the isomorphism of the strongly regular graphs which are involved.

\comment{ In fact, the graphs are isomorphic to the Brouwer-Wilbrink graphs (in the hyperbolic case) and the Metz graphs (in the elliptic case). For the hyperbolic case, this was conjectured in [3] and proved in [5]; for the elliptic case, this was conjectured for q = 4 in [9], and will be proved for general even q in [10].

 Two papers, one of Xiang and Hollmann and another of Penttila and Williford, both construct association schemes with the same parameters. In both papers there is an isomorphism question, in the former an isomorphism of strongly regular graphs, and in the latter of the 3 class association schemes from both papers.

Since the strongly regular graph is a relation of the 3 class scheme, the resolution of the latter question also resolves the former. 

While the paper of Xiang and Hollman announced a resolution to the question in a forthcoming paper, such a paper appears to have never been published, and since that was 14 years ago it is safe to say it is still open.

The paper in question resolves an open question in association schemes.
}

\end{document}